\documentclass[12pt, reqno]{amsart}
\usepackage[scaled]{helvet}
\usepackage{mathtools}
\mathtoolsset{showonlyrefs}
\usepackage{etoolbox}
\usepackage[utf8]{inputenc}
 \usepackage[T1]{fontenc}

\usepackage[pagewise]{lineno}
\usepackage{comment}
\usepackage{amsmath}
\usepackage{amssymb}
\usepackage{amsfonts}
\usepackage{mathrsfs}
\usepackage{amsbsy}
\usepackage{relsize}
\usepackage[colorlinks, linkcolor=fancy, citecolor=blue, urlcolor=Cblue, hypertexnames=false]{hyperref}
\usepackage{enumitem}
\usepackage{anyfontsize}
\usepackage{tikz}

\numberwithin{equation}{section}

\numberwithin{bbb}{section}

\usepackage[top=2.5cm, bottom=2.5cm, left=3cm, right=2.5cm]{geometry}

\definecolor{Cgrey}{rgb}{0.85,0.85,0.85}
\definecolor{Cblue}{rgb}{0.50,0.85,0.85}
\definecolor{Cred}{rgb}{1,.2,.4}
\definecolor{fancy}{rgb}{0.10,0.85,0.10}
\definecolor{mygreen}{rgb}{0.01,0.6,0.2}
\definecolor{tealgreen}{rgb}{0.0, 0.51, 0.5}
\definecolor{tangerine}{rgb}{0.95, 0.52, 0.0}
\definecolor{saffron}{rgb}{0.96, 0.77, 0.19}
\definecolor{mint}{rgb}{0.24, 0.71, 0.54}
\definecolor{lincolngreen}{rgb}{0.11, 0.35, 0.02}
\definecolor{lava}{rgb}{0.81, 0.06, 0.13}
\definecolor{lasallegreen}{rgb}{0.03, 0.47, 0.19}
\definecolor{mahogany}{rgb}{0.75, 0.25, 0.0}
\definecolor{electricultramarine}{rgb}{0.25, 0.0, 1.0}
\definecolor{mypink1}{rgb}{0.858, 0.188, 0.478}
\definecolor{mypink2}{RGB}{219, 48, 122}
\definecolor{mypink3}{cmyk}{0, 0.7808, 0.4429, 0.1412}
\definecolor{mygray}{gray}{0.6}
\definecolor{venetianred}{rgb}{0.78, 0.03, 0.08}
\definecolor{sapphire}{rgb}{0.03, 0.15, 0.4}
\definecolor{utahcrimson}{rgb}{0.83, 0.0, 0.25}
\definecolor{trueblue}{rgb}{0.0, 0.45, 0.81}
\definecolor{carminered}{rgb}{1.0, 0.0, 0.22}
\definecolor{cobalt}{rgb}{0.0, 0.28, 0.67}
\definecolor{cornflowerblue}{rgb}{0.39, 0.58, 0.93}
\definecolor{falured}{rgb}{0.5, 0.09, 0.09}

\usepackage{fancyhdr}
\usepackage{lipsum}

\newtheorem{theo}{Theorem}[section]
\newtheorem{prop}{Proposition}[section]
\newtheorem{lem}{Lemma}[section] 
\newtheorem{rem}{Remark}[section]
\newtheorem{defn}{Definition}[section]

\renewcommand{\geq}{\geqslant}
\renewcommand{\ge}{\geqslant}
\renewcommand{\leq}{\leqslant}
\renewcommand{\le}{\leqslant}

\usepackage{color}

\newcommand{\R}{\ensuremath{\mathbb{R}}}
\newcommand{\bq}{\mathbf{q}}
\newcommand{\bw}{\mathbf{w}}
\newcommand{\sv}{\mathsf{v}}
\newcommand{\bs}{\mathbf{S}}
\newcommand{\bb}{\mathbf{B}}
\newcommand{\bx}{\mathbf{X}}

\newlength{\absstep}
\usepackage{lipsum} 
\usepackage{tikz,xcolor}

\definecolor{lime}{HTML}{A6CE39}
\DeclareRobustCommand{\orcidicon}{
	\begin{tikzpicture}
	\draw[lime, fill=lime] (0,0) 
	circle [radius=0.16] 
	node[white] {{\fontfamily{qag}\selectfont \tiny ID}};
	\draw[white, fill=white] (-0.0625,0.095) 
	circle [radius=0.007];
	\end{tikzpicture}
	\hspace{-2mm}
}
\foreach \x in {A, B, C}{\expandafter\xdef\csname orcid\x\endcsname{\noexpand\href{https://orcid.org/\csname orcidauthor\x\endcsname}
			{\noexpand\orcidicon}}
}

\usepackage{bm}

\title[Fujita exponent]{Forcing Effects on Finite-Time Blow-Up in Degenerate and Singular Parabolic Equations}
\author[M. Majdoub \& B. T. Torebek]{Mohamed Majdoub\orcidA{} \& Berikbol T. Torebek\orcidB{}}
\address[M. Majdoub]{Department of Mathematics, College of Science, Imam Abdulrahman Bin Faisal University, P. O. Box 1982, Dammam, Saudi Arabia\newline   Basic and Applied Scientific Research Center, Imam Abdulrahman Bin Faisal University, P.O. Box 1982, 31441, Dammam, Saudi Arabia}
\email{\tt mmajdoub@iau.edu.sa}
\email{\tt med.majdoub@gmail.com}
\email{\tt  mohamed.majdoub@fst.rnu.tn}
\address[B. T. Torebek]{
Institute of Mathematics and Mathematical Modeling 28 Shevchenko Str., 050010 Almaty, Kazakhstan}
\email{\tt torebek@math.kz}

\begin{document}

\begin{abstract}
We study the degenerate and singular parabolic equation with a forcing term
\[
|x|^{\sigma_1}u_t = \Delta u + |x|^{\sigma_2}|u|^p + t^\varrho \mathbf{w}(x), \quad (t,x)\in(0,\infty)\times\mathbb{R}^N,
\]
where $N\ge 2$, $\sigma_1,\sigma_2>-2$, $\varrho>-1$, $p>1$, and $\mathbf{w}\in L^1(\mathbb{R}^N)$ is continuous. 
We establish critical exponents that sharply separate the regimes of global existence and finite-time blow-up. 
For $\varrho>0$, we prove that there is no weak global solution for all  $p>1$. 
When $-1<\varrho<0$, we show that if 
\[
p < p^*:=\frac{N+\sigma_2-\varrho(2+\sigma_1)}{N-2-\varrho(2+\sigma_1)},
\]
then every weak solution blows up in finite time, provided $\int\limits_{\mathbb{R}^N}\mathbf{w}(x)\,dx>0$. 
In the case $\varrho=0$, blow-up occurs for $p\le (N+\sigma_2)/(N-2)_+$ with $N\ge 2$. 
In contrast, for $p>p^*$ and under smallness conditions on the initial data and forcing term, we prove the existence of a unique global mild solution. 
The analysis relies on scaling transformations, semigroup estimates for degenerate operators, and a fixed-point argument in weighted-in-time Lebesgue spaces.\end{abstract}

\maketitle

\section{Introduction and main results}
In this paper, we investigate the critical behavior of the degenerate parabolic equation with forcing
\begin{equation}\label{main}
|x|^{\sigma_1} u_t=\Delta u+|x|^{\sigma_2}|u|^p+t^{\varrho}\bw(x),
\qquad (t,x)\in (0,\infty)\times\mathbb{R}^{N},
\end{equation}
where $N\ge1$, $\sigma_1,\sigma_2>-2$, $p>1$, $\varrho>-1$, and $\bw:\mathbb{R}^N\to\mathbb{R}$ is a continuous function given.  
We also consider the unforced version,
\begin{equation}\label{main-unforced}
|x|^{\sigma_1} u_t=\Delta u+|x|^{\sigma_2}|u|^p,
\qquad (t,x)\in (0,\infty)\times\mathbb{R}^{N}.
\end{equation}

Equation~\eqref{main-unforced} fits into the more general class of nonhomogeneous parabolic equations
\begin{equation}\label{main1}
\rho(x)u_t=\Delta u+\bq(x)|u|^{p},
\qquad (t,x)\in (0,\infty)\times\mathbb{R}^{N},
\end{equation}
where the positive function $\rho$ represents the density of the medium and $\bq$ is a prescribed continuous function.  
As explained in \cite{Pablo}, the model \eqref{main1} describes heat transfer in a nonhomogeneous medium of density $\rho$, subject to a nonlinear temperature–dependent source. Standard
 assumptions include positivity, boundedness, and smoothness of $\rho$, together with asymptotic behavior
\begin{equation}\label{rho-infinity}
\rho(x)\sim |x|^{-\sigma}\qquad \text{as } |x|\to\infty.
\end{equation}

The linear counterpart of~\eqref{main1},
\[
\rho(x)u_t=\Delta u,\qquad u(0,x)=u_0(x),
\]
was studied in \cite{Kamin2007}; see also \cite{Kamin1998, Punzo2009}.  
For the nonlinear problem with $\bq\equiv 1$, it is shown in \cite{Pablo} that when $p>1$ and $0<\sigma<2$, the Fujita exponent is
\[
p_F(N,\sigma)=1+\frac{2}{N-\sigma}.
\]
Under a stronger condition on $\rho$, blow-up occurs whenever the initial data satisfy (see \cite[Theorem~2.2]{Pablo})
\[
\lim_{|x|\to\infty} |x|^{\frac{\sigma}{p-1}}u_0(x)=\infty.
\]  
Further results for the cases $p=1$ and $0<p<1$ are also obtained in \cite{Pablo}.  
In \cite{Li-Xiang}, the equation \eqref{main1} is treated under the assumption $\bq(x)\sim |x|^{-m}$ as $|x|\to\infty$, $m\in\mathbb{R}$, leading to the Fujita exponent
\[
p_F(N,\sigma,m)=1+\frac{2-m}{N-\sigma}.
\]

In the radial case, \cite[Theorem~3.1]{Igar23} shows that the first and second critical Fujita exponents for \eqref{main-unforced} are
\begin{align}
\label{Fuj-exp1}
p_F &= 1+\frac{2+\sigma_2}{N+\sigma_1},\\
\label{Fuj-exp2}
\mu^* &= \frac{2(2+\sigma_2)}{(2+\sigma_1)(p-1)}.
\end{align}

The {\em  second critical} exponent $\mu^*>0$ is the threshold governing the influence of spatial decay of initial data:  
solutions blow up in finite time when $u_0(x)\sim C|x|^{-\mu}$ with $0<\mu<\mu^*$, while global solutions exist for $\mu>\mu^*$.

Although it is sometimes asserted that the first Fujita exponent for \eqref{main-unforced} follows immediately from scaling, this is misleading.  
Indeed, the equation is invariant under
\[
u_\lambda(t,x)
 = \lambda^{\frac{2+\sigma_2}{p-1}} 
   u\!\left(\lambda^{\sigma_1+2} t,\lambda x\right),
   \qquad \lambda>0,
\]
and the only Lebesgue space preserved by this scaling is $L^{p_c}(\mathbb{R}^N)$ with
\[
p_c=\frac{N(p-1)}{2+\sigma_2}.
\]
The condition $p_c=1$ corresponds to $p=1+\frac{2+\sigma_2}{N}$, which does \emph{not} match the true Fujita exponent~\eqref{Fuj-exp1}.  
Thus scaling alone does not reveal the correct threshold.

Following \cite{Igar23}, a key tool in the radial setting is the transformation
\begin{equation*}\label{Transf}
u(t,r)=\sv(\tau,z),
\qquad r=|x|,\quad z=r^\theta,\quad \tau=\Lambda t,
\end{equation*}
where $\Lambda>0$ and $\theta\in\mathbb{R}$ are chosen appropriately.  
Assuming $\bw(x)=\bw(r)$, a direct computation gives
\begin{equation}\label{main-rad}
\begin{split}
\Lambda\, \sv_\tau
&=\theta^2 z^{2-\frac{2+\sigma_1}{\theta}}\sv_{zz}
 +\theta(\theta+N-2)z^{1-\frac{2+\sigma_1}{\theta}}\sv_{z}
 +z^{\frac{\sigma_2-\sigma_1}{\theta}}|\sv|^p  \\
&\quad +\Lambda^{-\varrho}\tau^{\varrho} z^{-\frac{\sigma_1}{\theta}}
   \bw\!\left(z^{1/\theta}\right).
\end{split}
\end{equation}
Introduce the parameters
\[
\theta = 1 + \frac{\sigma_1}{2}, \qquad
\sigma = \frac{2(\sigma_2 - \sigma_1)}{2 + \sigma_1}, \qquad
\overline{N} = \frac{2(N + \sigma_1)}{2 + \sigma_1}, \qquad
\Lambda = \theta^{\frac{2\sigma}{2+\sigma}},
\]
and apply the change of variables
\begin{equation*}\label{Variab-s}
s=\theta^{-\frac{2}{2+\sigma}}z.
\end{equation*}
Then \eqref{main-rad} becomes
\begin{equation}\label{main-rad-final}
\sv_\tau = \sv_{ss} + \frac{\overline{N}-1}{s}\sv_{s} + s^\sigma |\sv|^p + \tau^{\varrho}\mathbf{W}(s),
\end{equation}
with
\begin{equation*}\label{W-s}
\mathbf{W}(s)
 = \Lambda^{-\varrho-1} s^{-\frac{2\sigma_1}{2+\sigma_1}}
   \bw\!\left(\theta^{\frac{2}{\theta(2+\sigma)}} s^{1/\theta}\right).
\end{equation*}
Equation~\eqref{main-rad-final} corresponds to the transformed version of
\begin{equation}\label{Transf-eq}
v_t = \Delta v + |x|^\sigma |v|^p + t^\varrho \bw(x)
\end{equation}
in radial coordinates for an effective spatial dimension $\overline{N}$.  
According to \cite{Ma}, the first Fujita exponent for \eqref{main} is
\begin{equation}\label{Fuj-main}
p^*=\frac{\overline{N}-2\varrho+\sigma}{\overline{N}-2\varrho-2}
\end{equation}
whenever $-1<\varrho\le0$ and $\displaystyle\int\limits_{\mathbb{R}^N}\bw(x)\,dx>0$.

The transformation leading to \eqref{Transf-eq} is introduced to shed light on the origin of the critical exponent and to place our problem within the framework of weighted Hardy--Hénon type equations, for which the critical behavior is well understood. This reformulation provides useful structural insight and helps motivate the identification of the threshold exponent.

However, we stress that equation~\eqref{Transf-eq} is not used in subsequent proofs. The preceding discussion and the transformation itself are intended solely as a heuristic argument to determine the critical exponent $p^*$.

In order to formulate the problem in a weak framework compatible with the weighted structure of \eqref{main}, we introduce the following notion of a weak global solution.
\begin{defn}\label{defn:weak-solution}
We say that $u$ is a weak global solution of \eqref{main}, with initial data $u(0,x)=u_0(x)$, if it satisfies
\begin{equation*}
\begin{split}
&u_0\in L^1_{\mathrm{loc}}(\mathbb{R}^N), \;
u\in L^1_{\mathrm{loc}}((0,\infty)\times\mathbb{R}^N),
\\
&|x|^{\sigma_2}|u|^p\in L^1_{\mathrm{loc}}((0,\infty)\times\mathbb{R}^N),
\quad
|x|^{\sigma_1}u \in L^1_{\mathrm{loc}}((0,\infty)\times\mathbb{R}^N),
\end{split}
\end{equation*}
and, for all $\psi\in C^{\infty}_{0}([0,\infty)\times\mathbb{R}^N)$,
\begin{equation*}
\begin{aligned}
\int\limits_{(0,\infty)\times\mathbb{R}^N}
\left(-|x|^{\sigma_1}\partial_t\psi+\Delta\psi\right)u\,dt\,dx
&=
\int\limits_{\mathbb{R}^N} |x|^{\sigma_1}u_0(x)\psi(0,x)\,dx  \\
&\quad
+\int\limits_{(0,\infty)\times\mathbb{R}^N}
\left(|x|^{\sigma_2}|u|^p + t^{\varrho}\mathbf w(x)\right)\psi\,dt\,dx 
\end{aligned}
\end{equation*}
\end{defn}

Our first main result establishes \eqref{Fuj-main} without requiring radial symmetry. More precisely, we have the following statement.
\begin{theo}\label{th1}
Let $N\geq 2$, $\sigma_1,\sigma_2>-2$, $p>1$, $u_{0}\in L^{1}_{\mathrm{loc}}(\mathbb{R}^{N})$, and $\mathbf{w}\in L^{1}(\mathbb{R}^{N})$. Moreover, assume that
\begin{equation*}\label{assum-w}
    \int\limits_{\mathbb{R}^{N}} \mathbf{w}(x)\,dx > 0.
\end{equation*}
Then the following assertions hold:
\begin{enumerate}[label=(\roman*)]
    \item If $\varrho>0$, then the problem~\eqref{main} admits no global-in-time weak solution for any $p>1$.

    \item If $-1<\varrho<0$ and
    \begin{equation}\label{crit}
        p < p^{*} 
        := \frac{N+\sigma_{2}-\varrho(2+\sigma_{1})}
        {\,N-2-\varrho(2+\sigma_{1})\,},
    \end{equation}
    then the problem~\eqref{main} admits no global-in-time weak solution.

    \item If $\varrho=0$, $N\geq 2$,  and
    \begin{equation*}\label{crit1}
        p \leq \frac{N+\sigma_{2}}{\,(N-2)_+\,},
    \end{equation*}
    then the problem~\eqref{main} admits no global-in-time weak solution.
\end{enumerate}
\end{theo}
\begin{rem}\label{rem:1.1} 
\rm 
    ~\begin{enumerate}[label=(\roman*)]
    \item For related results in the case $\sigma_1=\sigma_2=0$, we refer the reader to \cite{Opuscula, BLZ, JKS, Ma}. In particular,
\[
p^{*}=\frac{N-2\varrho}{N-2\varrho-2},
\]
which is consistent with \cite[Theorem 1.1, (1.8)]{JKS}.

\item  When $\varrho = 0$ and $N=2$, the critical exponent $p^*$ defined in \eqref{crit} becomes
$p^* = \infty.$ Consequently, in this case, there is no global-in-time weak solution for any $p>1$.

\item When $\varrho = 0$ and $\sigma_2 = 0$, the critical exponent $p^*$ defined in \eqref{crit} reduces to
\[
p^* = \frac{N}{N-2}, \quad N \ge 3,
\]
regardless of the value of $\sigma_1 > -2$. This coincides with the well-known Fujita exponent for the semilinear heat equation with a forcing term, which in the non-degenerate case $\sigma_1 = \sigma_2 = 0$ has been established in \cite{BLZ}. Consequently, for time-independent forcing ($\varrho = 0$) and an unweighted nonlinearity ($\sigma_2 = 0$), the degenerate or singular coefficient $|x|^{\sigma_1}$ does not influence the critical threshold for blow-up.

\item  In the critical case $\varrho=0$ and $p=\frac{N+\sigma_2}{N-2}$, the nonexistence result of Theorem~\ref{th1}(iii) extends the classical blow-up phenomenon for the Hardy–Sobolev type equations. The logarithmic cut-off function used in the proof is essential to handle the borderline decay.
\end{enumerate}
\end{rem}

We now turn to the global theory for exponents $p$ exceeding the critical value $p^*$ given in~\eqref{crit}.  
To this end, we study global-in-time mild solutions to the Cauchy problem
\[
u_t = |x|^{-\sigma_1}\Delta u
      + |x|^{\sigma_2 - \sigma_1}\,|u|^{p}
      + t^{\varrho}\,|x|^{-\sigma_1}\,\bw(x), 
      \qquad x \in \mathbb{R}^N,
\]
under suitable conditions on the parameters $\sigma_1$, $\sigma_2$, $\varrho$, and $p$, together with a smallness assumption on $u_0$ and $\bw$.  
Our analysis relies on the semigroup estimates for the operator $|x|^{-\sigma_1}\Delta$ established in Proposition~\ref{prop:weighted-estimate}, combined with a fixed-point argument in a weighted-in-time Lebesgue space.  
The precise statement is the following.
\begin{theo}\label{thm:global-main}
Let \(N\ge 2\) and assume
\[
-2\;<\;\sigma_2\;<\;\sigma_1\;\le\; 0,\qquad -1\;<\;\varrho\;<\;0,\qquad p\;>\;1.
\]
Define the critical exponents \(p_c, r_c, p^*\) by
\begin{equation}\label{pcritt}
\begin{split}
p_c \;&=\; \frac{N(p-1)}{2+\sigma_2},\\[2mm]
p^* \;&=\; \frac{N+\sigma_2 - \varrho(2+\sigma_1)}{\,N - 2 - \varrho(2+\sigma_1)\,},\\[2mm]
r_c \;&=\; \frac{N(p-1)}{2+\sigma_2 + (1+\varrho)(2+\sigma_1)(p-1)}.
\end{split}
\end{equation}
Assume, in addition, that
\begin{equation*}\label{p-p-star}
p > p^*.
\end{equation*}
Then there exists $\varepsilon>0$ such that, for every initial data 
$u_0$ and forcing term $\bw$ satisfying
\begin{equation}\label{small-assum}
\|u_0\|_{L^{p_c}} + \||\cdot|^{-\sigma_1}\bw\|_{L^{r_c}} < \varepsilon,
\end{equation}
the mild formulation
\begin{equation}\label{mild}
\nu(t)
= \bs(t)u_0
  + \int\limits_0^t \bs(t-s)\bigl(|x|^{\sigma_2-\sigma_1}\,|u(s)|^p\bigr)\,ds
  + \int\limits_0^t \bs(t-s)\bigl(s^{\varrho}\,|x|^{-\sigma_1} w\bigr)\,ds
\end{equation}
admits a unique global-in-time mild solution  
 $u\in L^\infty\!\bigl((0,\infty);L^{r}(\mathbb R^N)\bigr)$ satisfying\, $\displaystyle\sup_{t>0} t^\mu\|u(t)\|_{L^{r}}<\infty$, where $r$ and $\mu$ are defined in \eqref{choice-r} and \eqref{def:mu} below.
\end{theo}
\begin{rem}\label{rem:1.2}
\rm 
  ~\begin{enumerate}[label=(\roman*)]
    \item It is straightforward to verify that any mild solution of \eqref{mild} is also a weak solution. Consequently, all conclusions of Theorem~\ref{th1} apply directly to every mild solution.
    \item  Note that Theorem~\ref{thm:global-main} does not address the case $\varrho=0$, $p>\frac{N+\sigma_2}{N-2}$ and $N\geqslant 3$. A complete analysis of this regime lies beyond the scope of the theorem and requires separate investigation; see \cite{Zhang98} for the case $\sigma_1=\sigma_2=0$.
\item The smallness condition \eqref{small-assum} on the initial data and the forcing term is natural for the fixed-point argument. It would be interesting to investigate whether global solutions can exist for large data when $p>p^*$, or if blow-up may occur in a different manner. Furthermore, the critical case $p=p^*$ is left open and deserves further study.
\end{enumerate}
\end{rem}

Finally, we present a local existence result obtained by using a fixed point argument.
\begin{theo}
\label{thm:loc}
Let $N\ge2$ and assume that the parameters satisfy
\[
-2<\sigma_2\le \sigma_1\le0,\qquad -1<\varrho\le0,\qquad p>1.
\]
Let $q\ge1$ be such that
\begin{equation*}
    \label{assum-q}
q>\max\!\left\{\frac{Np}{N+\sigma_2},\ \frac{N(p-1)}{2+\sigma_2}\right\},
\qquad q\ge p.
\end{equation*}
Assume that the initial data and the forcing term satisfy
\[
u_0\in L^q(\mathbb{R}^N),
\qquad |x|^{-\sigma_1}\bw\in L^q(\mathbb{R}^N).
\]
Let $(\bs(t))_{t\ge0}$ be the $C_0$--semigroup generated  by the operator
$|x|^{-\sigma_1}\Delta$.
Then there exists a time $T>0$ such that the integral equation
\begin{equation}\label{eq:main}
u(t) = \mathbf{S}(t)u_0 
+ \int_0^t \mathbf{S}(t-s)\bigl(|x|^{\sigma_2-\sigma_1}\,|u(s)|^p\bigr)\,ds
+ \int_0^t \mathbf{S}(t-s)\bigl(s^{\varrho}\,|x|^{-\sigma_1} \bw(x)\bigr)\,ds.
\end{equation}
admits a unique mild solution
\[
u\in C\bigl([0,T];L^q(\mathbb{R}^N)\bigr).
\]
\end{theo}

\begin{rem}\label{rem:local}
\rm 
  ~\begin{enumerate}[label=(\roman*)]
    \item The unforced case, namely $\bw=0$, commonly referred to as the Hardy–Hénon equation, was studied in \cite{BTW} for $\sigma_1=0$, where the parameter $\gamma$ corresponds to $-\sigma_2$ in the present notation.
\item
The condition
\[
q>\frac{Np}{N+\sigma_2}
\]
ensures that the weighted semigroup estimate in Proposition~2.1 can be applied to the nonlinear term
$|x|^{\sigma_2-\sigma_1}|u|^p$ with $(q_1,q_2)=(q/p,q)$.

\item
The restriction
\[
q>\frac{N(p-1)}{2+\sigma_2}
\]
guarantees that the associated time singularity $(t-s)^{-\alpha}$ arising in the convolution kernel
is integrable, which is essential for the fixed-point argument. Here $\alpha$ is given by \eqref{alpha-loc} below.
\item
The assumption $q\ge p$ ensures that $|u|^p\in L^{q/p}(\mathbb{R}^N)$ whenever $u\in L^q(\mathbb{R}^N)$.

\item
Unlike the global existence result in Theorem~\ref{thm:global-main}, the present local theory allows the
borderline case $\sigma_2=\sigma_1$, since no decay in time is required for the nonlinear term.

\item
No sign condition on $\bw$ is imposed in the local existence result; only the integrability condition
$|x|^{-\sigma_1}\bw\in L^q(\mathbb{R}^N)$ is needed.

    \end{enumerate}
    \end{rem}
The article is organized as follows. In Section~\ref{sec:prelim}, we introduce the notation and collect the preliminary results needed throughout the paper. Section~\ref{sec:nonglobal} is devoted to the proof of Theorem~\ref{th1}, which establishes the nonexistence of global solutions. The global existence result is addressed in Section~\ref{sec:global}, where we prove Theorem~\ref{thm:global-main}. In Section~\ref{sec:local}, we establish the local existence result in Lebesgue spaces and provide the proof of Theorem~\ref{thm:loc}. Finally, Section~\ref{sec:conc} concludes the paper with a summary of the main results and a discussion of possible directions for future research.

Throughout the paper, the positive constant $C$ may change from line to line.

\section{Useful tools \& Auxiliary results}
\label{sec:prelim}
In this section, we introduce the notation used throughout the paper and present several auxiliary results and estimates. 

We begin by recalling several fundamental estimates for the heat semi-group
associated with the operator $|x|^{\alpha}\Delta$, where $\alpha\in(0,2)$.
This is relevant in our setting because equation~\eqref{main} can be rewritten as
\begin{equation}\label{main-bis}
    u_t = |x|^{\alpha}\Delta u 
          + |x|^{\sigma_2-\sigma_1}\,|u|^{p}
          + t^{\varrho}\,|x|^{-\sigma_1}\,\bw(x),
\end{equation}
with the identification $\alpha=-\sigma_1<2$.
Thus, the linear diffusion part in \eqref{main-bis} falls precisely within the
framework of the degenerate operator $|x|^{\alpha}\Delta$.
\begin{theo}[$L^a$–$L^b$ estimates for degenerate heat semigroups]\quad\\ \label{thm:hk-LpLq-alpha<2}
Let $N \ge 1$, $0 \le \alpha < 2$, and $1 < a, b < \infty$. Suppose the exponents $a,b$ satisfy
\begin{equation*}\label{Asumps}
    \begin{cases}
        \displaystyle \frac{1}{b} < \frac{1}{a} < 2 - \alpha, &\text{if } N = 1, \\[1.5ex]
        \displaystyle \frac{1}{b} < \frac{1}{a} < 1 - \frac{\alpha}{N}, &\text{if } N \ge 2.
    \end{cases}
\end{equation*}
Then the degenerate operator $L := |x|^{\alpha} \Delta$ generates a $C_0$-semigroup $(\mathbf{S}(t))_{t > 0}$, and there exists a constant $C = C(N, \alpha, a, b) > 0$ such that for all $t > 0$ and $\varphi \in L^a(\mathbb{R}^N)$,
\begin{equation}\label{eq:LpLq}
    \|\mathbf{S}(t)\varphi\|_{L^{b}(\mathbb{R}^N)} \le 
    C \, t^{\,-\frac{N}{2-\alpha}\left(\frac{1}{a} - \frac{1}{b}\right)} \|\varphi\|_{L^{a}(\mathbb{R}^N)}.
\end{equation}
\end{theo}
\begin{rem}
\rm 
   ~ \begin{enumerate}[label=(\roman*)]
        \item Theorem~\ref{thm:hk-LpLq-alpha<2} corresponds to \cite[Theorem~5.1]{CCM-2016}.
        
        \item The validity of the smoothing estimate \eqref{eq:LpLq} depends essentially on the parameter $\alpha$. 
        For the remaining cases—namely the singular case $\alpha < 0$, the borderline case $\alpha = 2$, and the super‑degenerate case $\alpha > 2$—a summary of the situation can be found in \cite{CCM-2016}.
        
        \item For further details and other heat kernel estimates, we refer to \cite{CCM-2016, Feller, Spina, Metaf-2012, Dirich-forms}.
    \end{enumerate}
\end{rem}
 We record below a smoothing estimate for the semigroup generated by the degenerate operator 
$|x|^{-\sigma_{1}}\Delta$, which will be used in the proof of Theorem~\ref{thm:global-main}.
\begin{prop}
\label{prop:weighted-estimate}
Let $N \geq 2$, $-2 < \sigma_1 \leq 0$, and let $\bs(t)$ be the $C_0$-semigroup generated by $|x|^{-\sigma_1}\Delta$. 
Assume $0 \leq \gamma < N$ and let $1 < q_1, q_2 < \infty$ satisfy
\begin{equation}\label{cond}
0 < \frac{1}{q_2} < \frac{\gamma}{N} + \frac{1}{q_1} < 1 + \frac{\sigma_1}{N}.
\end{equation}
Then there exists a constant $C > 0$ such that for every $t > 0$ and every $\varphi \in L^{q_1}(\mathbb{R}^N)$,
\begin{equation}\label{q-2-q-1}
\big\| \bs(t) (|\cdot|^{-\gamma} \varphi) \big\|_{L^{q_2}}
\le C \,
t^{-\frac{N}{2+\sigma_1}\left(\frac{1}{q_1} - \frac{1}{q_2}\right)
- \frac{\gamma}{2+\sigma_1}}
\,
\|\varphi\|_{L^{q_1}} .
\end{equation}
\end{prop}
The proof of Proposition~\ref{prop:weighted-estimate} follows the general 
strategy of \cite[Proposition~2.1]{BTW}, where the weight $|x|^{-\gamma}$ 
is decomposed and controlled by suitable Hölder estimates. In the present 
setting, the argument is further combined with the homogeneity properties 
and the $L^{p}$--$L^{q}$ bounds for the semigroup $(\bs(t))_{t>0}$ 
derived in \cite{CCM-2016}, which capture the scaling effects generated by 
the degenerate coefficient $|x|^{-\sigma_1}$. For completeness, a full 
proof is given below.

\begin{proof}[Proof of Proposition~\ref{prop:weighted-estimate}]
The proof proceeds in two steps: first we establish the estimate at $t = 1$, then we extend it to arbitrary $t > 0$ by a scaling argument.

\noindent
\textbf{Step 1: The case $t = 1$.} 
Set $m = N/\gamma$ (if $\gamma = 0$ the statement is trivial). 
Choose $\varepsilon, \delta > 0$ such that $\varepsilon < m$, and 
\begin{equation}
    \label{eq:eps-delt-1}
    \frac{1}{q_2} < \frac{1}{m+\delta} + \frac{1}{q_1} < \frac{1}{m-\varepsilon} + \frac{1}{q_1} < 1+\frac{\sigma_1}{N}.
\end{equation}
Such $\varepsilon, \delta$ exist because of \eqref{cond}. 
Decompose the potential as
\[
|\cdot|^{-\gamma} = \psi_1 + \psi_2,
\]
with $\psi_1 \in L^{m-\varepsilon}(\mathbb{R}^N)$, $\psi_2 \in L^{m+\delta}(\mathbb{R}^N)$.
Define
\begin{equation}
    \label{eq:eps-delt-2}
\frac{1}{r_1} = \frac{1}{m-\varepsilon} + \frac{1}{q_1}, \qquad
\frac{1}{r_2} = \frac{1}{m+\delta} + \frac{1}{q_1}.
\end{equation}
By construction, $0 < \frac{1}{r_1}, \frac{1}{r_2} < 1$ and $\frac{1}{q_2} < \frac{1}{r_2} < \frac{1}{r_1}$. 
Moreover, by taking $\varepsilon$ sufficiently small we can ensure that
\begin{equation}
    \label{eq:eps-delt-3}
\frac{1}{r_1} < 1 + \frac{\sigma_1}{N}.
\end{equation}

Combining \eqref{eq:eps-delt-1}, \eqref{eq:eps-delt-2}, and \eqref{eq:eps-delt-3}, we are in a position to apply the standard smoothing estimate \eqref{eq:LpLq} to the pairs $(a,b)=(q_2,r_1)$ and $(a,b)=(q_2,r_2)$. As a consequence, we obtain
\begin{equation}\label{smoothing}
\|\bs(1) f\|_{L^{q_2}} \le C \|f\|_{L^{r_i}}, 
\qquad i=1,2.
\end{equation}

Next, we decompose the weight and estimate
\[
\|\bs(1)(|\cdot|^{-\gamma}\varphi)\|_{q_2} 
\le 
\|\bs(1)(\psi_1 \varphi)\|_{q_2} 
+ 
\|\bs(1)(\psi_2 \varphi)\|_{q_2}.
\]
For $i = 1,2$, H\"older's inequality gives
\[
\|\psi_1 \varphi\|_{r_1} \le \|\psi_1\|_{L^{m-\varepsilon}}  \|\varphi\|_{q_1},\quad \|\psi_2 \varphi\|_{r_2} \le \|\psi_2\|_{L^{m+\delta}}  \|\varphi\|_{q_1}.
\]
Applying \eqref{smoothing} we obtain
\[
\|\bs(1)(\psi_1 \varphi)\|_{q_2} \le C \|\psi_1\|_{L^{m-\varepsilon}} \|\varphi\|_{q_1},\quad \|\bs(1)(\psi_2 \varphi)\|_{q_2} \le C \|\psi_2\|_{L^{m+\delta}} \|\varphi\|_{q_1}.
\]
Summing the two contributions yields
\begin{equation}\label{static}
\|\bs(1)(|\cdot|^{-\gamma} \varphi)\|_{q_2} \le C_1 \|\varphi\|_{q_1},
\end{equation}
with $C_1 = C(\|\psi_1\|_{L^{m-\varepsilon}} + \|\psi_2\|_{L^{m+\delta}})$.

\noindent
\textbf{Step 2: Scaling argument.}
Let $D_\lambda$ be the dilation operator $D_\lambda \varphi(x) = \varphi(\lambda x)$, $\lambda > 0$. 
One has $\|D_\lambda \varphi\|_r = \lambda^{-N/r} \|\varphi\|_r$ and
$D_\lambda(|\cdot|^{-\gamma} \varphi) = \lambda^{-\gamma} |\cdot|^{-\gamma} D_\lambda \varphi$.
A direct computation based on the homogeneity of the operator yields the following scaling property (see also \cite[Lemma 3.1]{CCM-2016}).
\begin{equation}\label{scaling}
D_\lambda^{-1} \bs(t) D_\lambda = \bs(\lambda^{2+\sigma_1} t).
\end{equation}
Take $\lambda = t^{-1/(2+\sigma_1)}$; then $\lambda^{2+\sigma_1} t = 1$ and \eqref{scaling} gives
\[
\bs(t) = D_\lambda \bs(1) D_\lambda^{-1}.
\]
Now compute
\begin{align*}
\bs(t)(|\cdot|^{-\gamma} \varphi)
&= D_\lambda \bs(1) D_\lambda^{-1} (|\cdot|^{-\gamma} \varphi) \\
&= D_\lambda \bs(1) \bigl( \lambda^{\gamma} |\cdot|^{-\gamma} D_\lambda^{-1} \varphi \bigr) \\
&= \lambda^{\gamma} D_\lambda \bigl[ \bs(1)(|\cdot|^{-\gamma} D_\lambda^{-1} \varphi) \bigr].
\end{align*}
Taking the $L^{q_2}$-norm,
\begin{align*}
\|\bs(t)(|\cdot|^{-\gamma} \varphi)\|_{q_2}
&= \lambda^{\gamma} \| D_\lambda [ \bs(1)(|\cdot|^{-\gamma} D_\lambda^{-1} \varphi) ] \|_{q_2}\\&
= \lambda^{\gamma} \lambda^{-N/q_2} \| \bs(1)(|\cdot|^{-\gamma} D_\lambda^{-1} \varphi) \|_{q_2}.
\end{align*}
Applying \eqref{static} to $D_\lambda^{-1} \varphi$,
\[
\| \bs(1)(|\cdot|^{-\gamma} D_\lambda^{-1} \varphi) \|_{q_2}
\le C_1 \| D_\lambda^{-1} \varphi \|_{q_1}
= C_1 \lambda^{N/q_1} \|\varphi\|_{q_1}.
\]
Hence,
\[
\|\bs(t)(|\cdot|^{-\gamma} \varphi)\|_{q_2}
\le C_1 \lambda^{\gamma + N(1/q_1 - 1/q_2)} \|\varphi\|_{q_1}.
\]
Substituting $\lambda = t^{-1/(2+\sigma_1)}$ we finally obtain
\begin{align*}
\|\bs(t)(|\cdot|^{-\gamma} \varphi)\|_{q_2}
&\le C_1 \, t^{-\frac{1}{2+\sigma_1}\bigl[ \gamma + N \bigl( \frac{1}{q_1} - \frac{1}{q_2} \bigr) \bigr]} \|\varphi\|_{q_1}\\&
= C_1 \, t^{-\frac{N}{2+\sigma_1}\bigl( \frac{1}{q_1} - \frac{1}{q_2} \bigr) - \frac{\gamma}{2+\sigma_1}} \|\varphi\|_{q_1}.
\end{align*}
Thus \eqref{q-2-q-1} holds with $C = C_1$.
\end{proof}

The next lemma plays a key role in building the functional setting needed to apply the fixed-point argument for establishing global existence of mild solutions.
\begin{lem}
\label{lem:f-p}
Let $N\ge 2$, $\sigma_{1}>-2$, $\sigma_{2}>-2$, and $-1<\varrho< 0$.  
Define 
\[
A:=2+\sigma_{1}>0,\qquad 
f(p)=\varrho A p^{2}-(N-2+\varrho A)p+(N+\sigma_{2}),
\]
and
\[
p^{*}=\frac{N+\sigma_{2}-\varrho A}{\,N-2-\varrho A\,}.
\]
Then 
\[
f(p)< 0\qquad\text{for all }p\ge p^{*}.
\]
\end{lem}
\begin{proof}[Proof of Lemma~\ref{lem:f-p}]
Since $\varrho<0$ and $A>0$, we have $\varrho A<0$, so the quadratic polynomial $f$ is concave.

We first verify that $f(p^{*})<0$ and $f'(p^{*})<0$. A straightforward substitution yields
\[
f(p^{*})=\frac{\varrho A\,(\sigma_{2}+2)^{2}}{(N-2-\varrho A)^{2}}<0,
\]
where the inequality follows immediately from $\varrho A<0$. Moreover,
\[
f'(p^{*})
   =\frac{2\varrho A (N+\sigma_{2})-(\varrho A)^{2}-(N-2)^{2}}
          {\,N-2-\varrho A\,}
   \equiv \frac{B}{\,N-2-\varrho A\,}.
\]
The numerator $B$ can be expressed as
\[
\begin{aligned}
B
 &= -\Big[(\varrho A)^{2}+(N-2)^{2}-2\varrho A (N+\sigma_{2})\Big] \\
 &= -\Big[(N-2-\varrho A)^{2}-2\varrho A (\sigma_{2}+2)\Big].
\end{aligned}
\]
Since $\sigma_{2}+2>0$ and $\varrho A<0$, the term $-2\varrho A(\sigma_{2}+2)$ is positive, which implies $B<0$ and therefore $f'(p^{*})<0$.

Because $f$ is concave and satisfies $f'(p^{*})<0$, it is nonincreasing on $[p^{*},\infty)$. Combined with $f(p^{*})<0$, we conclude that
\[
f(p)<0 \qquad \text{for all } p\ge p^{*}.
\]
This completes the proof.
\end{proof}

\section{Nonexistence of global solutions}
\label{sec:nonglobal}
This section presents the proof of Theorem~\ref{th1}, which proceeds via a case analysis.\\
\subsubsection{The case \texorpdfstring{$-1<\varrho< 0$ and $p<\frac{N+\sigma_2-\varrho(2+\sigma_1)}{N-2-\varrho(2+\sigma_1)}$}{-1<gamma<0 and p<(N+sigma2-gamma(2+sigma1))/(N-2-gamma(2+sigma1))}} The proof proceeds by contradiction. Assume that the problem \eqref{main} admits a weak global solution $u$.
We introduce a test function
\begin{equation*}\label{3}
\phi(t,x)\in C^1((0,\infty); C^2(\mathbb{R}^N)),
\end{equation*} defined by
$$
\phi(t,x)=\psi^{\frac{p}{p-1}}\left(\frac{t}{T}\right)\varphi^{\frac{2p}{p-1}}\left({\frac{|x|}{R}}\right),
$$
where the cutoff functions $\psi$ and $\phi$ are given by 
$$
\psi(s)=\begin{cases}
1 \hspace{0.12cm}\text{if}\hspace{0.12cm} 1/2\leq s\leq 3/4,\\
0 \hspace{0.12cm}\text{if}\hspace{0.12cm} s\in [0,1/4]\cup[4/5,\infty),
\end{cases}\quad\varphi(s)=\begin{cases}
1 \hspace{0.12cm}\text{if}\hspace{0.12cm}  s\in [0,1],\\
0 \hspace{0.12cm}\text{if}\hspace{0.12cm} s\geq 2.
\end{cases}
$$ Here $T>>1$ and $R>>1$ denote sufficiently large positive numbers.
This choice will be used to ensure that 
$$
\int\limits_{\mathbb{R}^N}u_0(x)\phi(0,x)dx=0.
$$
Therefore, from the Definition \ref{defn:weak-solution} of weak solutions we obtain
\begin{align*}
&\int\limits_0^T\int\limits_{\mathbb{R}^N}\left[|x|^{\sigma_2}|u(t,x)|^p+t^\varrho\mathbf{w}(x)\right]\phi(t,x)dxdt\\&\leq\int\limits_0^T\int\limits_{\mathbb{R}^N}\left( |u(t,x)||\Delta_x\phi(t,x)|+|x|^{\sigma_1}|u(t,x)||\phi_t(t,x)|\right)dxdt.
\end{align*}
By H\"{o}lder inequality together with the $\varepsilon$-Young inequality, we deduce
\begin{align*}
\int\limits_0^T\int\limits_{\mathbb{R}^N}|x|^{\sigma_1}|u||\phi_t|dxdt&\leq \left(\int\limits_0^T\int\limits_{\mathbb{R}^N}|x|^{\sigma_2}|u|^p\phi dxdt\right)^\frac{1}{p}\left(\int\limits_0^T\int\limits_{\mathbb{R}^N} {|\phi_t|^{\frac{p}{p-1}}}|x|^{\frac{\sigma_1 p-\sigma_2}{p-1}}{|\phi|^{-\frac{1}{p-1}}}dxdt\right)^{\frac{p-1}{p}}\\&\leq {\frac{1}{4}}\int\limits_0^T\int\limits_{\mathbb{R}^N} |x|^{\sigma_2}|u|^p dxdt +{C}\int\limits_0^T\int\limits_{\mathbb{R}^N}{|\phi_t|^{\frac{p}{p-1}}}|x|^{\frac{\sigma_1 p-\sigma_2}{p-1}}{|\phi|^{-\frac{1}{p-1}}}dxdt,
\end{align*}
and similarly,
\begin{align*}
\int\limits_0^T\int\limits_{\mathbb{R}^N}|u||\Delta_x\phi| dxdt&\leq \left(\int\limits_0^T\int\limits_{\mathbb{R}^N}|x|^{\sigma_2}|u|^p\phi dxdt\right)^\frac{1}{p}\left(\int\limits_0^T\int\limits_{\mathbb{R}^N}|x|^{-\frac{\sigma_2}{p-1}}{|\Delta_x \phi|^{\frac{p}{p-1}}}{|\phi|^{-\frac{1}{p-1}}}dxdt\right)^\frac{p-1}{p} \\&\leq{\frac{1}{4}}\int\limits_0^T\int\limits_{\mathbb{R}^N}|x|^{\sigma_2}|u|^p\phi dxdt+{C}\int\limits_0^T\int\limits_{\mathbb{R}^N}|x|^{-\frac{\sigma_2}{p-1}}{|\Delta_x \phi|^{\frac{p}{p-1}}}{|\phi|^{-\frac{1}{p-1}}}dxdt,
\end{align*} {where $C$ denotes a generic constant.}
{Hence, it follows that
\begin{equation}\label{4}
\begin{split}
\int\limits_0^T\int\limits_{\mathbb{R}^N}t^\varrho\mathbf{w}(x)\phi dxdt&\lesssim \int\limits_0^T\int\limits_{\mathbb{R}^N}{|\phi_t|^{\frac{p}{p-1}}}|x|^{\frac{\sigma_1 p-\sigma_2}{p-1}}{|\phi|^{-\frac{1}{p-1}}}dxdt \\&+\int\limits_0^T\int\limits_{\mathbb{R}^N}|x|^{-\frac{\sigma_2}{p-1}}{|\Delta_x \phi|^{\frac{p}{p-1}}}{|\phi|^{-\frac{1}{p-1}}}dxdt.
\end{split}
\end{equation}}
Straightforward computations yield the estimates
\begin{align*}
\left|\frac{d}{dt}\left(\psi^{\frac{p}{p-1}}\left(\frac{t}{T}\right)\right)\right|\leq CT^{-1}\psi^{\frac{1}{p-1}}\left(\frac{t}{T}\right)\left|\psi'\left(\frac{t}{T}\right)\right|,\end{align*} and \begin{align*}\left|\Delta\left(\varphi^{\frac{2p}{p-1}}\left(\frac{|x|}{R}\right)\right)\right|\leq CR^{-2}\phi^{\frac{2}{p-1}}\left(\frac{|x|}{R}\right).
\end{align*}
Using the properties of the cutoff functions $\Phi$ and $\Psi,$ together with the change of variables $t=T\tau$ and $x=Ry$ we have
\begin{equation}
\label{1*}\int_0^T\int_{\mathbb{R}^N}{|\phi_t|^{\frac{p}{p-1}}}|x|^{\frac{\sigma_1 p-\sigma_2}{p-1}}{|\phi|^{-\frac{1}{p-1}}}dxdt\leq
CT^{1-\frac{p}{p-1}}R^{\frac{\sigma_1p-\sigma_2}{p-1}+N},
\end{equation}
and
\begin{equation}\label{2*}
\int\limits_0^T\int\limits_{\mathbb{R}^N}|x|^{-\frac{\sigma_2}{p-1}}{|\Delta_x \phi|^{\frac{p}{p-1}}}{|\phi|^{-\frac{1}{p-1}}}dxdt\leq
CTR^{N-\frac{2p+\sigma_2}{p-1}}.\end{equation}
Hence, one can deduce that
\begin{equation}\label{3*}
\begin{split}
\int\limits_0^{T}t^{\varrho}\Psi^{\frac{p}{p-1}}\left(\frac{t}{T}\right)dt\int\limits_{\mathbb{R}^N}\mathbf{w}(x)\Phi^{\frac{2p}{p-1}}\left(\frac{|x|}{R}\right) dx= CT^{\varrho+1}\int\limits_{|x|<R}\mathbf{w}(x)\Phi^{\frac{2p}{p-1}}\left(\frac{|x|}{R}\right) dx.\end{split}
\end{equation} 
{Owing to \eqref{1*}-\eqref{3*}, we can rewrite the inequality \eqref{4} as
\begin{equation}
    \label{Final-ineq}
\int\limits_{\R^N}\mathbf{w}(x)\varphi^{\frac{2p}{p-1}}\left(\frac{|x|}{R}\right) dx\lesssim T^{-\varrho}R^{N-\frac{2p+\sigma_2}{p-1}}+T^{-\varrho-1-\frac{1}{p-1}}R^{\frac{\sigma_1p-\sigma_2}{p-1}+N}.
\end{equation}}
Taking $T=R^{\sigma_1+2}$, we can reformulate \eqref{Final-ineq} as
\begin{equation*}
    \label{Fin-ineq-R}
\int\limits_{\R^N}\mathbf{w}(x)\varphi^{\frac{2p}{p-1}}\left(\frac{|x|}{R}\right) dx\lesssim R^{N-2\varrho-\varrho\sigma_1-\frac{2p+\sigma_2}{p-1}}.    
\end{equation*}

Note that 
$$
N-2\varrho-\varrho\sigma_1-\frac{2p+\sigma_2}{p-1}<0 \quad \Longleftrightarrow\quad  p<\frac{N+\sigma_2-\varrho(2+\sigma_1)}{N-2-\varrho(2+\sigma_1)}.
$$
Hence, passing to the limit as $R \to +\infty$ in the last inequality, we have that
\begin{equation*}
\begin{split}
\lim\limits_{R \to +\infty}\int\limits_{|x|<R}\mathbf{w}(x)\Phi^{\frac{2p}{p-1}}\left(\frac{|x|}{R}\right) dx=\int\limits_{\mathbb{R}^N}\mathbf{w}(x) dx\leq 0,\end{split}\end{equation*}
which contradicts the fact that $\int\limits_{\mathbb{R}^N}\mathbf{w}(x)dx>0.$

\subsubsection{The case $\varrho=0,\,p\leq \frac{N+\sigma_2}{N-2}$}
$ $\\
$\bullet$ The case \texorpdfstring{$\varrho=0,\,p< \frac{N+\sigma_2}{N-2}$}{gamma=0, p < (N+sigma2)/(N-2-+sigma1)}.
The result follows immediately from the previous case by setting $\varrho=0.$\\
$\bullet$ {The case \texorpdfstring{$\varrho=0,\,p= \frac{N+\sigma_2}{N-2}$}{gamma=0, p = (N+sigma2)/(N-2-+sigma1)}}
Let us introduce a test function
\begin{equation*}\label{3-1}
\phi(t,x)\in C^1((0,\infty); C^2(\mathbb{R}^N)),
\end{equation*} defined by
$$
\phi(t,x)=\psi^{\frac{p}{p-1}}\left(\frac{t}{T}\right)\varphi^{\frac{2p}{p-1}}\left({\frac{\log\left(\frac{|x|}{\sqrt{R}}\right)}{\log(\sqrt{R})}}\right).
$$
The cutoff functions $\psi$ and 
$\phi$ are given by
$$
\psi(s)=\begin{cases}
1 \hspace{0.12cm}\text{if}\hspace{0.12cm} 1/2\leq s\leq 3/4,\\
0 \hspace{0.12cm}\text{if}\hspace{0.12cm} s\in [0,1/4]\cup[4/5,\infty),
\end{cases}\quad\varphi(s)=\begin{cases}
1 \hspace{0.12cm}\text{if}\hspace{0.12cm}  s\in (-\infty,0],\\
0 \hspace{0.12cm}\text{if}\hspace{0.12cm} s\geq 1.
\end{cases}
$$ Here $T>>1$ and $R>>1$ denote sufficiently large parameters.

By the properties of the function $\varphi$ one has (see \cite{BT22})
$$\left|\Delta\varphi^{\frac{2p}{p-1}}\left({\frac{\log\left(\frac{|x|}{\sqrt{R}}\right)}{\log(\sqrt{R})}}\right)\right|\leq C|x|^{-2}(\log R)^{-1}\varphi^{\frac{2}{p-1}},\,\sqrt{R}<|x|\leq R.$$
Consequently,
\begin{equation*}\label{2*-1}
\int\limits_0^T\int\limits_{\mathbb{R}^N}|x|^{-\frac{\sigma_2}{p-1}}{|\Delta_x \phi|^{\frac{p}{p-1}}}{|\phi|^{-\frac{1}{p-1}}}dxdt\leq
CT(\log R)^{\frac{2-N}{2+\sigma_2}}.\end{equation*}
Since $p= \frac{N+\sigma_2}{N-2},$ inequality \eqref{Final-ineq} can be replaced by
\begin{equation}
    \label{Final-ineq-1}
\int\limits_{\mathbb{R}^N}\mathbf{w}(x)\varphi^{\frac{2p}{p-1}}\left(\frac{\log\left(\frac{|x|}{\sqrt{R}}\right)}{\log (\sqrt{R})}\right) dx\lesssim (\log R)^{\frac{2-N}{2+\sigma_2}}+T^{-\frac{N+\sigma_2}{2+\sigma_2}}R^{\frac{\sigma_1N+\sigma_1\sigma_2+2\sigma_2+2N}{2+\sigma_2}}.
\end{equation}
By choosing $T=R^{m}$ with $m>\sigma_1+2$, inequality \eqref{Final-ineq-1} can be rewritten as
\begin{equation*}
    \label{Fin-ineq-R-1}
\int\limits_{\mathbb{R}^N}\mathbf{w}(x)\varphi^{\frac{2p}{p-1}}\left(\frac{\log\left(\frac{|x|}{\sqrt{R}}\right)}{\log (\sqrt{R})}\right) dx\lesssim (\log R)^{\frac{2-N}{2+\sigma_2}}+R^{-\frac{(N+\sigma_2)(m-\sigma_1-2)}{2+\sigma_2}}.    
\end{equation*}
Letting $R \to +\infty$ in the above estimate yields
\begin{equation*}
\begin{split}
\lim\limits_{R \to +\infty}\int\limits_{|x|<R}\mathbf{w}(x)\varphi^{\frac{2p}{p-1}}\left(\frac{\log\left(\frac{|x|}{\sqrt{R}}\right)}{\log (\sqrt{R})}\right) dx=\int\limits_{\mathbb{R}^N}\mathbf{w}(x) dx\leq 0.\end{split}\end{equation*}
This conclusion contradicts the assumption that $\int\limits_{\mathbb{R}^N}\mathbf{w}(x)dx>0.$

\subsubsection{The case $\varrho>0,\,p>1$}
Applying the estimates \eqref{1*}, \eqref{2*}, and \eqref{3*}, inequality \eqref{4} can be expressed as
\begin{equation*}
\begin{split}
\int\limits_{|x|<R}\mathbf{w}(x)\Phi^{\frac{2p}{p-1}}\left(\frac{|x|}{R}\right) dx\leq 
CT^{-\varrho-\frac{p}{p-1}}R^{\frac{\sigma_1p-\sigma_2}{p-1}+N}+CT^{-\varrho}R^{-\frac{2p+\sigma_2}{p-1}+N}.\end{split}\end{equation*}
Choosing $T=R^{m},$ the above estimate becomes
\begin{equation}\label{W1}
\begin{split}
\int\limits_{|x|<R}\mathbf{w}(x)\Phi^{\frac{2p}{p-1}}\left(\frac{|x|}{R}\right) dx\leq 
CR^{-\varrho m-\frac{pm}{p-1}+\frac{\sigma_1p-\sigma_2}{p-1}+N}+CR^{-\varrho m-\frac{2p+\sigma_2}{p-1}+N}.\end{split}\end{equation}
Now, taking $m=N\varrho^{-1},$ we observe that
$${-\varrho m-\frac{pm}{p-1}+\frac{\sigma_1p-\sigma_2}{p-1}+N}<0\quad\mbox{and}\quad {-\varrho m-\frac{2p+\sigma_2}{p-1}+N}<0.$$
Therefore both powers of $R$ on the right-hand side of \eqref{W1} are negative, and letting $R \to +\infty$ in \eqref{W1} gives
\begin{equation*}
\begin{split}
\lim\limits_{R \to +\infty}\int_{|x|<R}\mathbf{w}(x)\Phi^{\frac{2p}{p-1}}\left(\frac{|x|}{R}\right) dx=\int_{\mathbb{R}^N}\mathbf{w}(x) dx\leq 0,\end{split}\end{equation*}
which contradicts the assumption that $\displaystyle\int_{\mathbb{R}^N}\mathbf{w}(x)dx>0.$ This contradiction completes the proof.

\section{Global existence}
\label{sec:global}

To carry out the proof of Theorem~\ref{thm:global-main}, we will make use of the following technical lemma.

\begin{lem}\label{lem:choice-param}
Let \(N \ge 2\) and assume that the parameters satisfy
\begin{equation*}\label{param-assum}
    -2 < \sigma_2 < \sigma_1 \le 0, 
    \qquad -1 < \varrho < 0, 
    \qquad p > p^*,
\end{equation*}
where \(p^*\) is given in~\eqref{pcritt}. Then there exists \(r \in (1,\infty)\) such that
\begin{equation}\label{choice-r}
    \max\Bigl\{
        \frac{1}{p_c} - \frac{2+\sigma_1}{Np},\;
        \frac{1}{p_c} + \frac{\varrho(2+\sigma_1)}{N}
    \Bigr\}
    \;<\; \frac{1}{r} \;<\;
    \min\Bigl\{
        \frac{1}{p_c},\;
        \frac{N+\sigma_2}{Np}
    \Bigr\}.
\end{equation}

Define
\begin{equation}\label{def:mu}
    \mu := \frac{N}{2+\sigma_1}\Bigl(\frac{1}{p_c} - \frac{1}{r}\Bigr),
\end{equation}
\begin{equation}\label{def:beta}
    \beta := \frac{N}{2+\sigma_1}\Bigl(\frac{1}{r_c} - \frac{1}{r}\Bigr),
\end{equation}
and
\begin{equation}\label{delta-def}
    \delta
    := \frac{N(p-1)}{(2+\sigma_1)r}
       - \frac{\sigma_2 - \sigma_1}{2+\sigma_1},
\end{equation}
where $p_c, r_c$ are given by \eqref{pcritt}.
Then the quantities \(\mu\), \(\beta\), and \(\delta\) satisfy
\begin{equation}\label{mu-proper}
    0 < \mu < \frac{1}{p},
    \qquad
    0 < \beta < 1,
    \qquad
    0 < \delta < 1,
\end{equation}
and moreover,
\begin{equation}\label{mu-proper1}
    1 - p\mu - \delta
    = -\mu
    = \varrho + 1 - \beta.
\end{equation}
\end{lem}
The proof of Lemma~\ref{lem:choice-param} will be given later. For the moment, we simply record the required parameter choices and resume the main argument for Theorem~\ref{thm:global-main}.
\begin{proof}[Proof of Theorem~\ref{thm:global-main}]
\quad

For the parameters \(r\) and \(\mu\) specified in \eqref{choice-r} and \eqref{def:mu}, we introduce the Banach space
\[
\bx:=\Bigl\{ u\in L^\infty\!\bigl((0,\infty);L^{r}(\mathbb R^N)\bigr)\;:\;
\|u\|_{\bx}:=\sup_{t>0} t^\mu\|u(t)\|_{L^{r}}<\infty \Bigr\}.
\]

We seek a mild solution of \eqref{mild} inside the closed ball  
\(\overline{\bb}(0,2C\varepsilon)\subset X\), where \(\varepsilon>0\) will be chosen sufficiently small.

Define the operator \(G:\bx\to\bx\) by
\begin{equation*}\label{def:G}
    G(u)(t):=\bs(t)u_0 + F(u)(t) + H(t),
\end{equation*}
where
\begin{equation*}\label{def:F-H}
\begin{aligned}
    F(u)(t)
        &:= \int\limits_0^t \bs(t-s)\bigl(|x|^{\sigma_2-\sigma_1}|u(s)|^{p}\bigr)\,ds,\\[2mm]
    H(t)
        &:= \int\limits_0^t \bs(t-s)\bigl(s^{\varrho}|x|^{-\sigma_1}\bw\bigr)\,ds .
\end{aligned}
\end{equation*}

Since \(p>p^*\) and \eqref{choice-r} imply
\[
\frac1r < \frac1{p_c} < 1+\frac{\sigma_1}{N},
\]
we may apply the unweighted estimate \eqref{q-2-q-1} with \(q_1=p_c\), \(q_2=r\), and use \eqref{def:mu} to obtain
\[
\|\bs(t)u_0\|_{L^{r}}
    \le C\,t^{-\frac{N}{2+\sigma_1}\left(\frac1{p_c}-\frac1r\right)}
          \|u_0\|_{L^{p_c}}
    = C\,t^{-\mu}\|u_0\|_{L^{p_c}}.
\]
Thus,
\begin{equation}\label{linear-bound}
    \sup_{t>0} t^\mu \|\bs(t)u_0\|_{L^{r}}
        \le C\,\|u_0\|_{L^{p_c}}.
\end{equation}

For the nonlinear term \(F(u)\), we apply the weighted estimate \eqref{q-2-q-1} with
\(q_1=r/p\), \(q_2=r\), and \(\gamma=\sigma_1-\sigma_2>0\), yielding
\begin{equation*}\label{nl-bound1}
    \|F(u)(t)\|_{L^r}
    \le C\int\limits_0^t (t-s)^{-\delta}\|u(s)\|_{L^r}^p\,ds,
\end{equation*}
where \(\delta\) is given by \eqref{delta-def}.

Multiplying by \(t^\mu\), using the definition of \(\|u\|_{\bx}\) and \eqref{mu-proper}, we find
\begin{equation*}\label{nl-bound2}
\begin{aligned}
t^\mu\|F(u)(t)\|_{L^r}
&\le C\|u\|_{\bx}^p\,t^\mu\!\int\limits_0^t (t-s)^{-\delta}s^{-p\mu}\,ds\\
&\le C\|u\|_{\bx}^p\,t^{\mu+1-p\mu-\delta}\!\int\limits_0^1 (1-\tau)^{-\delta}\tau^{-p\mu}\,d\tau,
\end{aligned}
\end{equation*}
and the last integral is finite by \eqref{mu-proper}. Hence
\begin{equation}\label{nonlinear-bound}
    \sup_{t>0} t^\mu\|F(u)(t)\|_{L^r}
    \le C\|u\|_{\bx}^p.
\end{equation}

A standard Lipschitz estimate based on  
\(|\,|a|^p-|b|^p|\le C(|a|^{p-1}+|b|^{p-1})|a-b|\) gives
\begin{equation}\label{nonlinear-lip}
\|F(u)-F(v)\|_{\bx}
\le C\big(\|u\|_{\bx}^{p-1}+\|v\|_{\bx}^{p-1}\big)\|u-v\|_{\bx}.
\end{equation}

Next, applying the unweighted estimate \eqref{q-2-q-1} with \(q_1=r\), \(q_2=r_c\), we obtain
\begin{equation*}\label{forc-bound1}
\begin{aligned}
\|H(t)\|_{L^r}
&\le C\left\||\cdot|^{-\sigma_1}\bw\right\|_{L^{r_c}}
   \int\limits_0^t (t-s)^{-\beta}s^\varrho\,ds \\
&\le C\left\||\cdot|^{-\sigma_1}\bw\right\|_{L^{r_c}}
   t^{\varrho+1-\beta}\!\int\limits_0^1 (1-\tau)^{-\beta}\tau^\varrho\,d\tau,
\end{aligned}
\end{equation*}
where \(\beta\) is defined in \eqref{def:beta}.  
Using \eqref{mu-proper}–\eqref{mu-proper1} and \(\varrho>-1\), the last integral is finite, and thus
\begin{equation}\label{forcing-bound}
\sup_{t>0} t^\mu\|H(t)\|_{L^r}
\le C\left\||\cdot|^{-\sigma_1}\bw\right\|_{L^{r_c}}.
\end{equation}

Combining \eqref{linear-bound}, \eqref{nonlinear-bound}, and \eqref{forcing-bound}, we obtain for every
\(u\in\overline{\bb}(0,2C\varepsilon)\),
\begin{equation*}\label{stab-bound}
\begin{aligned}
\|G(u)\|_{\bx}
&\le C\bigl(\|u_0\|_{L^{p_c}}+\left\||\cdot|^{-\sigma_1}\bw\right\|_{L^{r_c}}\bigr)
   + C\|u\|_{\bx}^p\\
&\le C\varepsilon + C(2C\varepsilon)^p\\
&\le 2C\varepsilon,
\end{aligned}
\end{equation*}
provided \(\varepsilon>0\) is chosen sufficiently small.  
Thus \(G\) maps \(\overline{\bb}(0,2C\varepsilon)\) into itself whenever  
\(\|u_0\|_{L^{p_c}}+\left\||\cdot|^{-\sigma_1}\bw\right\|_{L^{r_c}} \le \varepsilon\).

Finally, by \eqref{nonlinear-lip}, for all \(u,v\in \overline{\bb}(0,2C\varepsilon)\),
\begin{equation*}\label{contrac-bound}
\|G(u)-G(v)\|_{\bx}
\le 2C(2C\varepsilon)^{p-1}\|u-v\|_{\bx}.
\end{equation*}
Choosing \(\varepsilon>0\) small enough so that  
\(2C(2C\varepsilon)^{p-1}<1\), the operator \(G\) becomes a contraction on  
\(\overline{\bb}(0,2C\varepsilon)\).  
By the Banach fixed-point theorem, \(G\) admits a unique fixed point  
\(u\in\overline{\bb}(0,2C\varepsilon)\), which is the desired global mild solution of \eqref{mild} in \(\bx\).

This completes the proof of Theorem~\ref{thm:global-main}.
\end{proof}

Next, we prove Lemma~\ref{lem:choice-param}.
\begin{proof}[Proof of Lemma~\ref{lem:choice-param}]
Set \(A:=2+\sigma_{1}>0\).  
The assumption \(p>p^{*}\) immediately yields
\begin{equation*}\label{first-observ}
    \frac1{p_{c}}
    \;<\;
    \frac{N-2-\varrho A}{N}
    \;<\;
    1+\frac{\sigma_{1}}{N}
    \;\le\; 1.
\end{equation*}
The proof is divided into three steps.

\medskip
\noindent\textbf{1. Existence of \(r\).}
We need to verify that
\begin{equation}\label{exist-r}
    \max\Bigl\{
        \tfrac1{p_{c}}-\tfrac{A}{Np},\;
        \tfrac1{p_{c}}+\tfrac{\varrho A}{N}
    \Bigr\}
    \;<\;
    \min\Bigl\{
        \tfrac1{p_{c}},\;
        \tfrac{N+\sigma_{2}}{Np}
    \Bigr\}.
\end{equation}
Since \(\varrho<0\), both
\[
\frac1{p_{c}}+\frac{\varrho A}{N}
    <\frac1{p_{c}},
\qquad
\frac1{p_{c}}-\frac{A}{Np}
    <\frac1{p_{c}},
\]
hold automatically. Hence the left-hand side of \eqref{exist-r} is strictly smaller than \(1/p_{c}\).

From Lemma~\ref{lem:f-p}, we know that for all \(p\ge p^{*}\),
\[
f(p)
=\varrho A p^{2}-(N-2+\varrho A)p+(N+\sigma_{2})<0.
\]
The inequality
\[
\frac1{p_{c}}+\frac{\varrho A}{N}
    < \frac{N+\sigma_{2}}{Np}
\]
is equivalent to \(f(p)<0\); hence it holds whenever \(p>p^{*}\).

Next, the inequality
\[
\frac1{p_{c}}-\frac{A}{Np}
    <\frac{N+\sigma_{2}}{Np}
\]
is equivalent to
\[
p>\frac{N+2+\sigma_{1}+\sigma_{2}}{N+\sigma_{1}}.
\]
A direct computation shows that
\[
p^{*}
\;>\;
\frac{N+2+\sigma_{1}+\sigma_{2}}{N+\sigma_{1}},
\]
so this inequality is also satisfied for all \(p>p^{*}\).

\medskip
\noindent\textbf{2. Bounds on \(\mu,\beta,\delta\).}
Since \(1/r<1/p_{c}\), we have \(\mu>0\).  
The inequality
\[
\frac1r
> \frac1{p_{c}}-\frac{A}{Np}
\]
gives
\[
\mu
=\frac{N}{A}\Bigl(\frac1{p_{c}}-\frac1r\Bigr)
<\frac1p.
\]

For
\(
\beta=\frac{N}{A}\!\left(\frac1{r_{c}}-\frac1r\right),
\)
we use the identity
\begin{equation*}\label{ident-rc-pc}
    \frac1{r_{c}}
    =\frac1{p_{c}}+\frac{(1+\varrho)A}{N}.
\end{equation*}
Since \(\varrho>-1\), we have \(1/r_{c}>1/p_{c}>1/r\), hence \(\beta>0\).  

Moreover,
\[
\frac1r
> \frac1{p_{c}}+\frac{\varrho A}{N}
= \frac1{r_{c}}-\frac{A}{N},
\]
which implies \(\beta<1\).  
Thus \(0<\beta<1\).

To control \(\delta\), use the identity
\[
\frac1r = \frac1{p_{c}}-\frac{A}{N}\mu.
\]
Substituting into
\[
\delta
=\frac{N(p-1)}{Ar}
 +\frac{\sigma_{1}-\sigma_{2}}{A}
\]
gives
\begin{align*}
\delta
&=\frac{N(p-1)}{A}
    \Bigl(\frac1{p_{c}}-\frac{A}{N}\mu\Bigr)
    +\frac{\sigma_{1}-\sigma_{2}}{A} \\
&=\frac{2+\sigma_{2}}{A}
  +\frac{\sigma_{1}-\sigma_{2}}{A}
  -(p-1)\mu \\
&=1-(p-1)\mu.
\end{align*}
Since \(0<\mu<1/p\), we have \((p-1)\mu<1-1/p\), hence
\[
0<\delta<1.
\]

\medskip
\noindent\textbf{3. Algebraic identities.}
Using \(\delta=1-(p-1)\mu\), we obtain
\begin{align*}
1-p\mu-\delta
&=1-p\mu-\bigl(1-(p-1)\mu\bigr)
\\&=-\mu.\end{align*}

For the second identity,
\begin{align*}
\mu-\beta
&= \frac{N}{A}
  \Bigl(\frac1{p_{c}}-\frac1{r_{c}}\Bigr)\\&
= \frac{N}{A}
  \Bigl(-\frac{(1+\varrho)A}{N}\Bigr)\\&
= -(1+\varrho),
\end{align*}
hence
\[
-\mu=\varrho+1-\beta.
\]

This completes the proof of Lemma~\ref{lem:choice-param}.
\end{proof}
\section{Local existence}
\label{sec:local}
This section is devoted to the proof of Theorem~\ref{thm:loc}. 
The argument is based on a fixed-point method in the Banach space
\[
X_T := \Bigl\{ u \in C\bigl([0,T]; L^q(\mathbb{R}^N)\bigr) : \|u\|_{X_T} < \infty \Bigr\},
\]
endowed with the norm
\[
\|u\|_{X_T} := \sup_{0 \le t \le T} \|u(t)\|_{L^q(\mathbb{R}^N)}.
\]

We introduce the operator
\[
F(u)(t) = \mathbf{S}(t)u_0 
+ \int_0^t \mathbf{S}(t-s)\bigl(|x|^{\sigma_2-\sigma_1}\,|u(s)|^p\bigr)\,ds
+ \int_0^t \mathbf{S}(t-s)\bigl(s^{\varrho}\,|x|^{-\sigma_1}\bw\bigr)\,ds.
\]
Our goal is to show that, for a suitable choice of $T>0$ and for an appropriate radius $M>0$, the map $F$ is a contraction on the closed ball
\[
B_M := \bigl\{ u \in X_T : \|u\|_{X_T} \le M \bigr\}.
\]

Since $\mathbf{S}(t)$ is a $C_0$-semigroup on $L^q(\mathbb{R}^N)$, there exists a constant $C_0>0$ such that
\[
\|\mathbf{S}(t)u_0\|_{L^q} \le C_0 \|u_0\|_{L^q}, 
\qquad t \ge 0.
\]

For the nonlinear term, we apply Proposition~1 with
\[
\gamma = \sigma_1-\sigma_2 >0, 
\qquad q_2 = q, 
\qquad q_1 = \frac{q}{p}.
\]
Under this choice, condition~\eqref{cond} reads
\[
0 < \frac{1}{q} < \frac{\sigma_1-\sigma_2}{N} + \frac{p}{q} < 1 + \frac{\sigma_1}{N}.
\]
The right-hand inequality implies $q > \frac{Np}{N+\sigma_2}$, which is ensured by our assumptions. 
The corresponding time-decay exponent given by~\eqref{q-2-q-1} is
\begin{equation}
    \label{alpha-loc}
    \alpha 
:= \frac{N}{2+\sigma_1}\Bigl(\frac{p}{q}-\frac{1}{q}\Bigr)
   + \frac{\sigma_1-\sigma_2}{2+\sigma_1}
 = \frac{N(p-1)}{q(2+\sigma_1)} + \frac{\sigma_1-\sigma_2}{2+\sigma_1}.
\end{equation}
The integrability condition $\alpha<1$ is equivalent to
\[
\frac{N(p-1)}{q} < 2+\sigma_2,
\]
that is, $q > \frac{N(p-1)}{2+\sigma_2}$, which again follows from the choice of $q$.

Concerning the forcing term, we set
\[
f := |x|^{-\sigma_1}\bw \in L^q(\mathbb{R}^N),
\]
so that the forcing takes the form $s^{\varrho}f(x)$. 
Applying Proposition~\ref{prop:weighted-estimate} with $\gamma=0$ and $q_1=q_2=q$ requires
\[
0 < \frac{1}{q} < 1 + \frac{\sigma_1}{N},
\qquad\text{equivalently}\qquad
 \frac{N}{N+\sigma_1}<q<\infty.
\]
This condition is automatically satisfied since $p>1$ and $\sigma_2<\sigma_1$ imply
\[
\frac{Np}{N+\sigma_2} > \frac{N}{N+\sigma_1}.
\]
Hence,
\[
\bigl\| \mathbf{S}(t-s)(s^{\varrho}f) \bigr\|_{L^q}
\le C\, s^{\varrho} \|f\|_{L^q}.
\]
In what follows, the constant $C$ denotes a generic positive constant that may vary from line to line.

Let $u\in X_T$. 
The linear term satisfies
\[
\|\mathbf{S}(t)u_0\|_{L^q} \le C_0 \|u_0\|_{L^q}.
\]
For the nonlinear contribution, Proposition~\ref{prop:weighted-estimate} yields
\[
\bigl\|\mathbf{S}(t-s)\bigl(|x|^{\sigma_2-\sigma_1}|u(s)|^p\bigr)\bigr\|_{L^q}
\le C (t-s)^{-\alpha}\|u(s)\|_{L^q}^p,
\]
and therefore
\[
\Bigl\|\int_0^t \mathbf{S}(t-s)\bigl(|x|^{\sigma_2-\sigma_1}|u(s)|^p\bigr)\,ds\Bigr\|_{L^q}
\le \frac{C}{1-\alpha}\, t^{1-\alpha}\|u\|_{X_T}^p.
\]
For the forcing term we obtain
\[
\Bigl\|\int_0^t \mathbf{S}(t-s)\bigl(s^{\varrho}f\bigr)\,ds\Bigr\|_{L^q}
\le \frac{C}{\varrho+1}\, t^{\varrho+1}\|f\|_{L^q}.
\]

We now fix
\[
M := 2C_0\|u_0\|_{L^q},
\qquad
R(T) := C_1 T^{1-\alpha} M^p + C_2 \|f\|_{L^q} T^{\varrho+1},
\]
with $C_1 = \frac{C}{1-\alpha}$ and $C_2 = \frac{C}{\varrho+1}$. 
Choosing $T>0$ sufficiently small so that $R(T)\le \frac{M}{2}$, which is possible since $1-\alpha>0$ and $\varrho>-1$, we obtain for all $u\in B_M$ and $t\in[0,T]$,
\[
\|F(u)(t)\|_{L^q}
\le C_0\|u_0\|_{L^q} + R(T)
\le M.
\]
Hence, $F(B_M)\subset B_M$.

To prove the contraction property, let $u,v\in B_M$. 
Using the inequality
\[
\bigl||u|^p-|v|^p\bigr|
\le p\bigl(|u|^{p-1}+|v|^{p-1}\bigr)|u-v|
\]
together with H\"older's inequality, we deduce
\[
\|\,|u|^p-|v|^p\|_{L^{q/p}}
\le 2p M^{p-1}\|u-v\|_{L^q}.
\]
Applying Proposition~\ref{prop:weighted-estimate} once again yields
\[
\bigl\|\mathbf{S}(t-s)\bigl(|x|^{\sigma_2-\sigma_1}(|u|^p-|v|^p)\bigr)\bigr\|_{L^q}
\le 2pC M^{p-1}(t-s)^{-\alpha}\|u-v\|_{L^q}.
\]
Consequently,
\[
\|F(u)(t)-F(v)(t)\|_{L^q}
\le \frac{2pC}{1-\alpha}\, M^{p-1} t^{1-\alpha}\|u-v\|_{X_T},
\]
and hence
\[
\|F(u)-F(v)\|_{X_T}
\le C_3 M^{p-1} T^{1-\alpha}\|u-v\|_{X_T},
\qquad
C_3 := \frac{2pC}{1-\alpha}.
\]
By possibly reducing $T$ further so that $C_3 M^{p-1} T^{1-\alpha}<1$, the map $F$ becomes a contraction on $B_M$.

The Banach fixed point theorem then guarantees the existence of a unique fixed point of $F$ in $B_M$, which is exactly a mild solution $u \in C\bigl([0,T]; L^q(\mathbb{R}^N)\bigr)$
of the integral equation~\eqref{eq:main}.

\section{Concluding remarks}
\label{sec:conc}

This work is devoted to the analysis of the critical behavior of a degenerate parabolic equation with a space--time-dependent forcing term, namely equation~\eqref{main}. We establish the existence of a sharp threshold that separates finite-time blow-up from global-in-time existence, characterized by the critical exponent $p^{*}$ defined in~\eqref{crit}.

The analysis combines scaling arguments, semigroup estimates associated with the degenerate operator $|x|^{-\sigma_1}\Delta$, and a fixed-point method in suitably time-weighted Lebesgue spaces. The resulting critical exponent $p^{*}$ extends the classical Fujita exponent and reflects the subtle interplay between the degeneracy parameters $\sigma_1$ and $\sigma_2$, the temporal growth rate $\varrho$, and the spatial dimension $N$.

Several natural questions remain open. In particular, the behavior at the critical threshold $p=p^{*}$ is not covered by the present results and is expected to require more refined analytical tools. The stationary forcing case $\varrho=0$, as well as the possibility of global existence for large initial data when $p>p^{*}$, also deserve further investigation. Moreover, allowing sign-changing or non-radial forcing terms may significantly alter the dynamics and lead to new critical phenomena. Extending the analysis to more general degenerate or singular weights beyond the range $\sigma_1,\sigma_2>-2$ is another challenge.

Overall, this study contributes to the qualitative theory of degenerate parabolic equations with time-dependent forcing. The explicit form of the critical exponent $p^{*}$ highlights how temporal growth, spatial degeneracy, and weighted nonlinearities interact. The techniques developed here, particularly the use of weighted semigroup estimates combined with fixed-point arguments, are flexible and may be adapted to other nonlinear evolution equations with variable coefficients and external sources.\\

    


\noindent$\rule[0.05cm]{16.2cm}{0.05cm}$

\noindent{\bf\large Declarations.}
On behalf of all authors, the corresponding author states that there is no conflict of interest. 
No data-sets were generated or analyzed during the current study.\\

\noindent{\bf\large Funding.}
 BT is supported by the Science Committee of the Ministry of Education and Science of the Republic of Kazakhstan (Grant No. AP23483960).

\noindent$\rule[0.05cm]{16.2cm}{0.05cm}$


\end{document}